\documentclass{amsart}
\pdfoutput=1 % ensure arXiv processes with pdflatex

\usepackage[margin=1in]{geometry}
\usepackage[semibold]{sourceserifpro}
\usepackage{amsmath,amssymb,amsthm}
\usepackage{tikz-cd}

\newcommand{\word}[1]{\textbf{#1}} % for highlighting things when we define them
\usepackage{varwidth} % make text boxes have more equal-length lines
\usepackage[capitalize]{cleveref}
\crefname{enumi}{}{} % don't write "Item" before enumerate items
\creflabelformat{enumi}{(#2#1#3)} % put parens around enum items when referencing them
\theoremstyle{definition}
\newtheorem{dfn}{Definition}[section]

\theoremstyle{plain}
\newtheorem{thm}[dfn]{Theorem}
\newtheorem{zbthm}[dfn]{Example/Theorem}
\newtheorem{prop}[dfn]{Proposition}
\newtheorem{lem}[dfn]{Lemma}
\newtheorem{cor}[dfn]{Corollary}
\crefname{cor}{Corollary}{Corollaries}

\theoremstyle{remark}
\newtheorem{zbbackend}[dfn]{Example}
\newenvironment{zb}[1][]{\begin{zbbackend}[#1]}{\hspace*{\fill}$\diamondsuit$\end{zbbackend}}
\newtheorem{nzbbackend}[dfn]{Non-example}

\newtheorem{rmk}[dfn]{Remark}

\usepackage[
backend=biber,
style=alphabetic,
maxnames=100
]{biblatex}
\DeclareMathOperator{\rk}{rk} % matroid rank function

\newcommand{\clan}[1]{\overline{#1}^{\mathrm{an}}} % analytic closure
\newcommand{\K}{\mathbb{K}} % field
\renewcommand{\L}{\mathcal{L}} % las vergnas face lattice, i.e. the lattice of acyclic flats
\renewcommand{\P}{\mathbb{P}} % projective space

\newcommand{\R}{\mathbb{R}} % real numbers
\newcommand{\SL}{\operatorname{SL}} % special linear group

\RequirePackage{xspace}
\def\a{\alpha} % alpha
\def\b{\beta} % beta
\newcommand{\CP}{\ensuremath{\mathcal{P}}\xspace}
\newcommand{\BR}{\ensuremath{\mathbb {R}}\xspace}

\usepackage{xcolor}
\usepackage{graphicx}

\title{Total positivity for matroid Schubert varieties}
\author[Xuhua He]{Xuhua He}
\address{Department of Mathematics and New Cornerstone Science Laboratory, The University of Hong Kong, Pokfulam, Hong Kong, Hong Kong SAR, China}
\email{xuhuahe@hku.hk}
\author[Connor Simpson]{Connor Simpson}
\address{Department of Mathematics, University of Wisconsin---Madison, 480 Lincoln Drive, Madison WI 53706-1388, USA}
\email{csimpson6@wisc.edu}
\author{Kaitao Xie}
\address{Department of Mathematics and New Cornerstone Science Laboratory, The University of Hong Kong, Pokfulam, Hong Kong, Hong Kong SAR, China}
\email{kaitaoxie@connect.hku.hk}

\begin{document}
\maketitle
\begin{abstract}
  % avoid using custom commands here to make it easy to paste into arXiv later
  We define the totally nonnegative matroid Schubert variety $\mathcal Y_V$ of a linear subspace $V \subset \mathbb R^n$.
  We show that $\mathcal Y_V$ is a regular CW complex homeomorphic to a closed ball, with strata indexed by pairs of acyclic flats of the oriented matroid of $V$. This closely resembles the regularity theorem for totally nonnegative generalized flag varieties.
  As a corollary, we obtain a regular CW structure on the real matroid Schubert variety of $V$.
\end{abstract}
\section{Introduction}

\subsection{Matroids}\label{intro:matroids}
Matroids model the combinatorics of linear subspaces, and have found broad application in and out of mathematics  \cite{rst18,recski10, iri83} since their formulation by Nakasawa \cite{Nak} and Whitney \cite{Whi}.
They enjoy a particularly close relationship with algebraic geometry \cite{katz14,ardila21}.

In this work we study the so-called ``matroid Schubert varieties''.
If $V \subset \K^n$ is a linear subspace, then its \word{matroid Schubert variety} $Y_V$ is the Zariski closure of $V$ in $(\P_\K^1)^n$, which contains $\K^n$ as an open subset.
Introduced by \cite{AB16}, matroid Schubert varieties are central to the proof of the Top Heavy Conjecture for realizable matroids \cite{HW17},  guide the conjecture's resolution for \textit{all} matroids \cite{BHMPW20b}, and  are the geometric model for matroidal Kazhdan-Lusztig theory \cite{EPW16}.
Preceding \cite{AB16}, a neighborhood of the most singular point of a matroid Schubert variety was studied in \cite{PS06} and \cite{T02}.

The geometry of $Y_V$ is controlled by the \word{flats} of $V$; that is, the sets $F \subset \{1, \ldots, n\}$ such that there is $v \in V$ whose zero coordinates are exactly those indexed by $F$. The flats of $V$ are an example of a \word{matroid}.
When $\K = \R$, we may consider the more refined notion of \word{covectors}, which record the combinations of signs that the coordinate functions of $\R^n$ can take on $V$.
This data gives an example of an \word{oriented matroid}.
Our main theorem says that oriented matroid data controls the geometry of the \word{totally non-negative matroid Schubert variety} $\mathcal Y_V := \clan{V \cap \R_{\geq 0}^n}$, the closure of $V \cap \R_{\geq 0}^n$ in $(\P_\R^1)^n$ with respect to the analytic topology.

\subsection{Total positivity} 
By definition, an invertible real matrix is called \word{totally positive} if all the minors are positive and  \word{totally non-negative} if all the minors are non-negative. These notions were introduced in the 1930s by Schoenberg~\cite{Sch}.  The theory of totally positive real matrices was further developed by Whitney and Loewner in the 1950s and found important applications in many different areas, including, for example, statistics, game theory, mathematical economics, and stochastic processes. We refer to the book by Karlin~\cite{Ka}  for detailed discussions. 

All  $n \times n$ invertible matrices form the general linear group, which is an example of a split reductive group. The theory of total positivity for an arbitrary split real reductive group was developed by Lusztig in his foundational work~\cite{Lus-1} and has had significant impacts on many active research directions, including, among others,
\begin{itemize}
	\item the theory of cluster algebras by Fomin and Zelevinsky~\cite{FZ},
	
	\item higher Teichm\"uller theory by Fock and Goncharov~\cite{FG}, 
	
	\item the theory of the amplituhedron by Arkani-Hamed and Trnka~\cite{AHT}. 
	
\end{itemize}	

It has also been discovered that many spaces with  $G$-action have natural positive structures. A typical example is the (partial) flag variety $\mathcal P$. This has a natural decomposition into (open) Richardson varieties: $\mathcal P=\sqcup_\a \mathcal P_{\a}$. This is a stratification, i.e., the closure of each $\CP_\a$ (under the Zariski topology) is a disjoint union of other Richardson varieties $\CP_\b$. On the other hand, Lusztig defined the totally non-negative flag $\mathcal P_{\ge 0}$. This is a semi-algebraic subvariety of $\mathcal B$. We then have the decomposition $$\mathcal P_{\ge 0}=\bigsqcup_{\a} \mathcal P_{\a, >0}, \quad \text{ where } \mathcal P_{>0}=\mathcal P_{\ge 0} \cap \mathcal P_{\a}.$$
Lusztig refers to the totally non-negative flag as a ``remarkable polyhedral space''. It has been studied by many leading experts: Bao, Galashin, Karp, Lam, Lusztig, Marsh, Postnikov, Rietsch, Williams, the first-named author, and others. They have established many remarkable geometric/topological properties, including the following: 
\begin{itemize}
	\item {\color{blue} Connected components}: $\mathcal P_{\a, >0}$ is a connected component of $\mathcal P_\a (\BR)$.
	
	\item {\color{blue} Cell structure}: $\mathcal P_{\a, >0} \cong \mathbb R_{>0}^{\dim \mathcal P_\a}$ is a semi-algebraic cell. 
	
	\item {\color{blue} Cellular decomposition}:  $\clan{\mathcal P_{\a, >0}}$ is a disjoint union of other totally positive cells $\mathcal P_{\beta, >0}$. 
	
	\item {\color{blue} Regularity property}: $\clan{\mathcal P_{\a, >0}}$ is a regular CW complex homeomorphic to a closed ball.
\end{itemize}

\subsection{Main result}\label{sec:mainresult}
One may expect that matroid Schubert varieties admits a ``nice'' positive structure, similar to the flag varieties. This is what we will establish in this paper.

\medskip
Let $E$ be a finite set.
If $V \subset \R^E$ is a linear subspace, then $Y_V \subset (\P_\R^1)^n$ can be decomposed as a disjoint union of locally closed ``Richardson varieties'' $Y_{FG}^\circ := Y_V \cap (0^F \times \R_{\neq 0}^{G \setminus F} \times \infty^{E \setminus F})$, with $F \subset G \subset E$ running over all flats of $V$.
For any sets $F \subset G \subset E$, we analogously define $\mathcal Y_{FG}^\circ := \mathcal Y_V \cap (0^F \times \R_{>0}^{G \setminus F} \times \infty^{E \setminus G})$, and let $\mathcal Y_{FG} := \clan{\mathcal Y_{FG}^\circ}$. Note that $\mathcal Y_{\emptyset,E} = \mathcal Y_V$ by definition.
Call a flat $F$ of $V$ \word{acyclic} if $V \cap (0^F \times \R_{>0}^{E \setminus F})$ is nonempty.
The \word{rank} of a flat is the codimension in $V$ of the subspace $V \cap \{x_i = 0 : i \in F\}$.

The main result of this paper is that the totally non-negative matroid Schubert variety is a ``remarkable polyhedral space''. More precisely, 
\begin{thm}\label{schubertthm}
  Let $V \subset \R^E$, with matroid Schubert variety $Y_V$ and totally non-negative Schubert variety $\mathcal Y_V$.
\begin{enumerate}
\item $\mathcal Y_{FG}^\circ$ is nonempty if and only if $F \subset G$ are acyclic flats of $V$.
  In this case, $\mathcal Y_{FG}^\circ$ is a single connected component of $Y_{FG}^\circ$, and is a semi-algebraic cell isomorphic to $(\R_{>0})^{\rk(G) - \rk(F)}$. \label{main:strata}

\item The closure $\mathcal Y_{FG}$ of a nonempty cell $\mathcal Y_{FG}^\circ$ decomposes as the disjoint union of cells $\mathcal Y_{F',G'}^\circ$ with $F \subset F' \subset G' \subset G$.
  \label{main:boundary}
\item This decomposition makes $\mathcal Y_{FG}$ a regular CW complex homeomorphic to a closed ball.\label{main:ball}
\end{enumerate}
\end{thm}
Some comparison is due.
Combinatorially, we see a new phenomenon in the matroid setting.
The cells of $\CP_{\geq 0}$ and $\mathcal Y_V$ are obtained by intersecting these sets with real Richardson strata of $\CP$ and $Y_V$, respectively.
The poset of boundary strata of $\CP$ is \word{thin}---that is, every interval of length two has exactly four elements---and $\CP_{\geq 0}$ contains exactly one connected component of every stratum. Hence, the poset of cells in the boundary of $\CP_{\geq 0}$ is also thin, a fact which is helpful for establishing the regularity property.
On the other hand, the poset of boundary strata of $Y_V$ is \emph{not} thin.
However, $\mathcal Y_V$ fails to meet all strata of $Y_V$, and surprisingly, its cell poset  \emph{is} thin. As in the Lie-theoretic setting, this fact helps us to establish regularity.

Geometrically, the Richardson strata of matroid Schubert varieties are simpler than those of flag varieties.
In the matroid Schubert case, each Richardson stratum is a hyperplane arrangement complement.
Every connected component of a real hyperplane arrangement complement is homeomorphic to an open ball.
However, a real open Richardson variety in a flag variety may have connected components with nontrivial topology (see, e.g. \cite{MR00}).
The relative simplicity of the matroid case's geometry allows us to show that $\mathcal Y_V$ is a ball by directly exhibiting it as a cone over a closed ball in its boundary, bypassing such high-powered tools as the Poincar\'e conjecture, which underpins the known proofs of \cref{schubertthm}'s Lie-theoretic analogues. 

\begin{zb}
    Let $V \subset \R^5$ be the linear subspace cut out by
    \begin{gather*}
      x_1+x_2 - x_3 = x_3 - x_4 - x_5  = 0.
    \end{gather*}
    The poset of flats of $V$ (below left) is \emph{not} thin, so its interval poset, which indexes strata of $Y_V$, is also \emph{not} thin.
    On the other hand, the subposet of acyclic flats (below right) \emph{is} thin, so its interval poset, which indexes cells of $\mathcal Y_V$ is also thin.
\begin{center}
\begin{minipage}{0.4\linewidth}
\begin{tikzpicture}[scale=0.6,
        every node/.style={font=\footnotesize, circle, draw=blue, minimum size=17pt, inner sep=1pt}]%%% poset of flats
\node (1) at(5,8){$E$};
\node (2) at(0,6){$123$};
\node (3) at(2,6){$14$};
\node (4) at(4,6){$24$};
\node (5) at(6,6){$25$};
\node (6) at(8,6){$15$};
\node (7) at(10,6){$345$};
\node (8) at(1,4){$1$};
\node (9) at(3,4){$2$};
\node (10) at(5,4){$3$};
\node (11) at(7,4){$4$};
\node (12) at(9,4){$5$};
\node (13) at(5,2){$\emptyset$};

\draw[-] (1)--(2);
\draw[-] (1)--(3);
\draw[-] (1)--(4);
\draw[-] (1)--(5);
\draw[-] (1)--(6);
\draw[-] (1)--(7);
\draw[-] (2)--(8);
\draw[-] (2)--(9);
\draw[-] (2)--(10);
\draw[-] (3)--(8);
\draw[-] (3)--(11);
\draw[-] (4)--(9);
\draw[-] (4)--(11);
\draw[-] (5)--(8);
\draw[-] (5)--(12);
\draw[-] (6)--(9);
\draw[-] (6)--(12);
\draw[-] (7)--(10);
\draw[-] (7)--(11);
\draw[-] (7)--(12);
\draw[-] (8)--(13);
\draw[-] (9)--(13);
\draw[-] (10)--(13);
\draw[-] (11)--(13);
\draw[-] (12)--(13);
\end{tikzpicture}%%%%%% end poset of flats
\end{minipage}
\hspace{3em}
\begin{minipage}{0.3\linewidth}
\begin{tikzpicture}[scale=0.6,
        every node/.style={font=\footnotesize, circle, draw=blue, minimum size=17pt, inner sep=1pt}]%%%%%% poset of acyclic flats
\node (1) at(3,8){$E$};
%\node (2) at(0,6){$123$};
\node (3) at(0,6){$14$};
\node (4) at(2,6){$24$};
\node (5) at(4,6){$15$};
\node (6) at(6,6){$25$};
%\node (7) at(10,6){$345$};
\node (8) at(0,4){$1$};
\node (9) at(2,4){$2$};
%\node (10) at(5,4){$3$};
\node (11) at(4,4){$4$};
\node (12) at(6,4){$5$};
\node (13) at(3,2){$\emptyset$};

%\draw[-] (1)--(2);
\draw[-] (1)--(3);
\draw[-] (1)--(4);
\draw[-] (1)--(5);
\draw[-] (1)--(6);
%\draw[-] (1)--(7);
%\draw[-] (2)--(8);
%\draw[-] (2)--(9);
%\draw[-] (2)--(10);
\draw[-] (3)--(8);
\draw[-] (3)--(11);
\draw[-] (4)--(9);
\draw[-] (4)--(11);
\draw[-] (5)--(8);
\draw[-] (5)--(12);
\draw[-] (6)--(9);
\draw[-] (6)--(12);
%\draw[-] (7)--(10);
%\draw[-] (7)--(11);
%\draw[-] (7)--(12);
\draw[-] (8)--(13);
\draw[-] (9)--(13);
%\draw[-] (10)--(13);
\draw[-] (11)--(13);
\draw[-] (12)--(13);
\end{tikzpicture}%%%%% end poset of acyclic flats
\end{minipage}
\end{center}

\smallskip
The non-negative matroid Schubert variety $\mathcal Y_V$ (below) is homeomorphic to a closed 3-ball. Cells of $\mathcal Y_V$ are indexed by intervals in the poset of acyclic flats, ordered by inclusion. Hence, the cells structure of $\mathcal Y_V$ has ten 0-cells and sixteen 1-cells (labelled), along with eight 2-cells and one 3-cell (unlabelled).
One sees immediately that the closure of any cell is homeomorphic to a closed ball, so the cell structure is regular.

\smallskip
\begin{center}
  \includegraphics[width=2.6in]{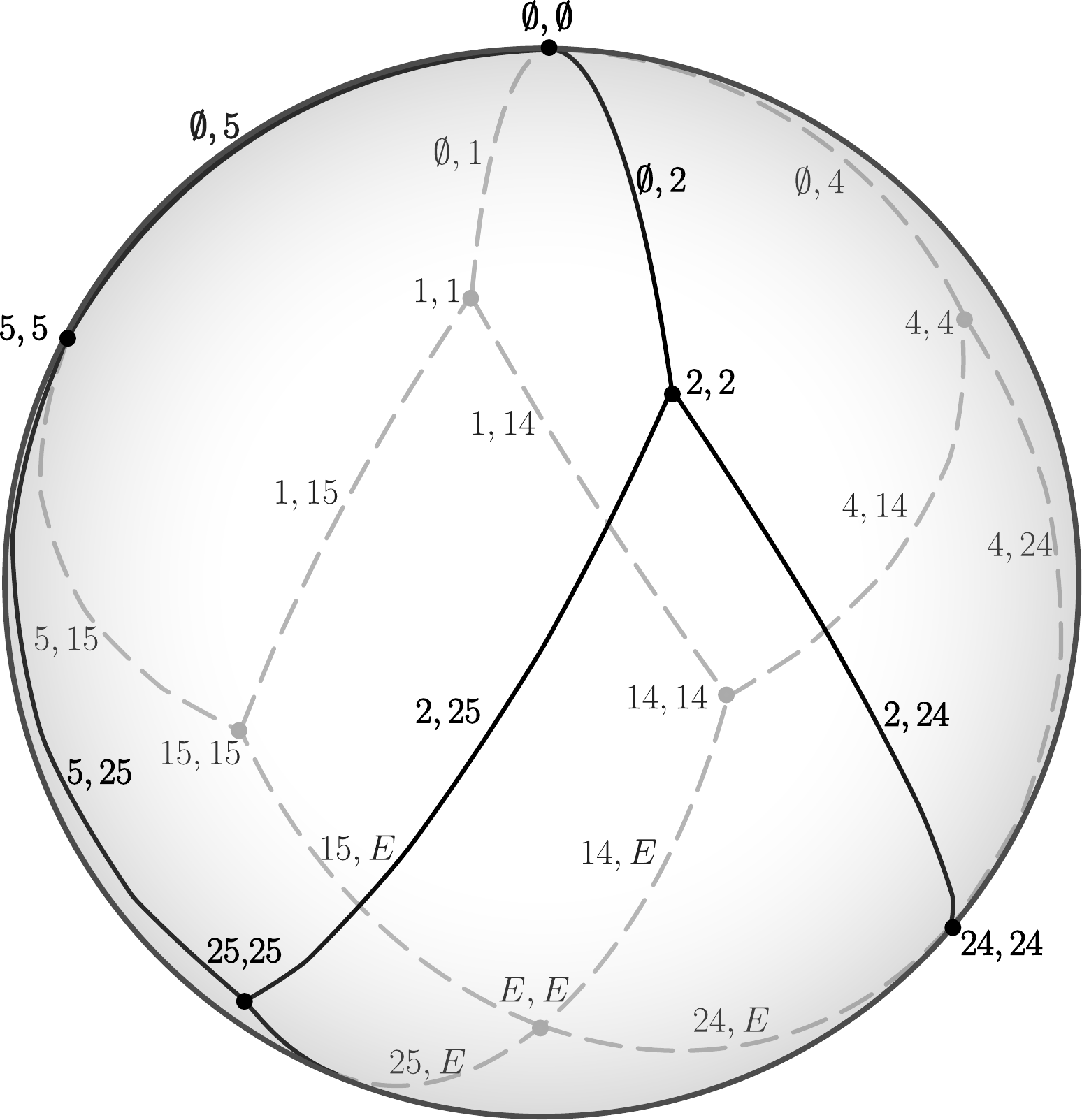}
\end{center}
\end{zb}

\subsection*{Acknowledgements}
XH is partially supported by the New Cornerstone Foundation through the New Cornerstone Investigator Program and the Xplorer prize, and by Hong Kong RGC Grant 14300023. 
CS thanks: Xuhua for his support \& hospitality during the visits that enabled this collaboration, Botong Wang for helpful conversation, and both CUHK and BICMR for pleasant working environments.

\section{Matroids and oriented matroids}
We review aspects of (oriented) matroid theory, comprehensively explored in \cite{W86} and \cite{bvswz99}, and state the main properties of matroid Schubert varieties.

\subsection{}
We may omit braces when writing one-element sets, e.g. ``$\{1,2\} \cup i$'' means ``$\{1,2\} \cup \{i\}$'' and ``$\{1,2\} \times 0$'' means ``$\{1,2\} \times \{0\}$''.
If $E$ and $K$ are sets, with $E$ finite, then $K^E := \prod_{i \in E} K$.
If $F \subset E$, then $\pi_F: K^E \to K^F$ is the projection.
If all factors in a product are single-element sets, then we may omit notation for the product, e.g. ``$\{0\} \times \{1\} \times \{1\}$'' will be written $0^{\{1\}}1^{\{2,3\}}$, and $0^E$ represents the origin of $\R^E$.
Both conventions on singletons will be violated as necessary to avoid confusion.

\smallskip
\noindent Throughout this paper, $E$ will denote a finite set.
\smallskip

In addition to sets, we will need to work with \word{signed sets}; that is, elements of $\{-,0,+\}^E$.
If $X$ is a signed set, write $X^-$, $X^0$, and $X^+$ for the coordinates of $X$ that have value $-$, $0$, and $+$, respectively.
We define $-X$ to be the signed set with $(-X)^- := X^+$, $(-X)^0 := X^0$, and $(-X)^+ := X^-$.
If $X$ and $Y$ are signed sets, then their \word{composition} is given by
\[
  (X\circ Y)_i := \begin{cases} X_i, & \text{if $X_i \neq 0$} \\ Y_i, & \text{otherwise.}\end{cases}
\]
Say  $X$ is \word{contained in}  $Y$, and write $X \leq Y$, if $X^+ \subset Y^+$ and $X^- \subset Y^-$.

For most terminology on posets, we refer to \cite{EC1}.
The \word{opposite} of a poset $P$ is $P^{op}$, the poset on the same underlying set as $P$, but with order reversed.

\subsection{}\label{bg:matroids}
A \word{matroid} on a finite set $E$ is defined by a collection of \word{flats} $F \subset E$ such that (i) $E$ is a flat, (ii) the intersection of two flats is a flat, and (iii) if $F$ is a flat and $i \in E \setminus F$, then there is a unique smallest flat containing $F \cup i$. The flats of a linear subspace, defined in \cref{intro:matroids}, satisfy these properties, giving us a recipe for producing a matroid from a linear subspace.

When ordered by inclusion, the flats of a matroid $\underline M$ form a graded lattice.
The \word{rank} of $\underline M$, denoted $\rk(\underline M)$, is the length of any maximal chain in this poset.
More generally, the \word{rank} of a flat $F$ of $\underline M$ is the length of any maximal chain of flats contained in $F$, and is denoted $\rk(F)$.
The \word{loops} of $\underline M$ are the elements of the minimal flat of $\underline M$. Call $\underline M$ \word{loopless} if its minimum flat is empty.

If $F \subset E$ is a flat of $\underline M$, then we can form the \word{restriction}  $\underline M|_F$ and \word{contraction} $\underline M/F$, matroids on $F$ and $E$, respectively, with flats
\[
  \{ G \subset F : \text{$G$ is a flat of $\underline M$}\}
  \quad \text{and}\quad
  \{ G \supset F : \text{$G \supset F$ is a flat of $\underline M$}\}.
\]

\begin{rmk}[Matroids and linear algebra]
  If $\K$ is a field and $V \subset \K^E$ is a linear subspace, then the flats of $V$ defined in \cref{intro:matroids} are the flats of a matroid $\underline M$. The rank of $\underline M$ is $\dim V$.
  Any matroid that arises in this manner is \word{realizable}, and $V$ is its \word{realization}.

  Let $\pi_F: \K^E \to \K^F$ be the coordinate projection.
  The restriction of $\underline M$ to $F$ is realized by $\pi_F(V) \subset \K^F$, and the contraction $\underline M/F$ is realized by $V \cap \ker(\pi_F)$.
  The element $i \in E$ is a loop of $\underline M$ if and only if $\pi_i(V) = \{0\}$.
  \end{rmk}

\medskip

Let $\K$ be a field and $V \subset \K^E$.
Recall (from \cref{intro:matroids}) that the \word{matroid Schubert variety} $Y_V$ associated to a linear subspace $V \subset \K^E$ is the Zariski closure of $V$ in $(\K \cup \infty)^E = (\P_\K^1)^E$.
For each pair of flats $F \subset G$ of $V$, let $Y_{FG}^\circ := Y_V \cap (0^F \times (\K_{\neq 0})^{G \setminus F} \times \infty^{E \setminus G})$, and let $Y_{FG}$ be the Zariski closure of $Y_{FG}^\circ$.

\begin{thm}\cite[Section 7]{PXY18}\label{schubertfacts}
  Let $\K$ be a field.
  Let $V \subset \K^E$ be a linear subspace, with associated matroid $\underline M$.
  \begin{enumerate}
  \item The intersection $Y_V \cap (\K^F \times \infty^{E \setminus F})$ is nonempty if and only if $F$ is a flat, in which case the intersection is equal to $\pi_F(V) \times \infty^{E \setminus F}$.\label{zstrata}
 \item If $F$ is a flat,  then $Y_V \cap \big((\P^1_\K)^F \times \infty^{E \setminus F}\big) = Y_{\pi_F(V)} \times \infty^{E \setminus F}$. \label{ziso}
 \item If $F$ is a flat, then $Y_V \cap \big( 0^F \times (\P^1_\K)^{E \setminus F}\big) = Y_{V \cap \ker(\pi_F)}$.\label{zcontract}
 \item $Y_{FG}$ is the disjoint union of all $Y_{F'G'}^\circ$ with $F \subset F' \subset G' \subset G$. \label{zclosurerel}
\end{enumerate}
\end{thm}
If $L$ is the set of loops of $V$'s matroid, then $Y_V = 0^L \times Y_{\pi_{E \setminus L}(V)}$, so we lose little by assuming the matroid of $V$ is loopless.

\subsection{}\label{bg:oriented}
An \word{oriented matroid} $M$ on a finite set $E$ is the data of a collection of \word{covectors} $\mathcal C \subset \{-,0,+\}^E$ such that
\begin{enumerate}
\item $0^E \in \mathcal C$,
\item $\mathcal C$ is closed under composition and negation
\item If $X,Y \in \mathcal C$ and $X(i) = -Y(i) \neq 0$, then there exists $Z \in \mathcal C$ such that $Z(i) = 0$ and $Z(j) = (X \circ Y)_j =(Y \circ X)_j$ for all $j$ such that $X_j = Y_j$.
\end{enumerate}
The above axioms imply the collection $\{ X^0 : X \in \mathcal C\}$ is the flats of a matroid $\underline M$, the \word{underlying matroid} of $M$.
\word{Flats} and \word{loops} of an oriented matroid are those of its underlying matroid.

Ordering $\{-,0,+\}$ by $0 < -$ and $0 < +$,  we induce a partial order on $\mathcal C$.
The poset $\mathcal C \cup \{\hat 1\}$, formed by adjoining a maximum to $\mathcal C$, is a graded lattice.
% The rank of a covector is $\rk(X) = \rk_{\underline M}(\underline M) - \rk_{\underline M}(X^0)$, where $\rk_{\underline M}$ is the rank function of the underlying matroid.
Maximal covectors are called \word{topes}.

Fix $A \subset E$.
By negating in each covector the coordinates indexed by $A$, we obtain a new subset $\mathcal C' \subset \{-,0,+\}^E$. In fact, $\mathcal C'$ the covectors of an oriented matroid $M'$, called a \word{reorientation} of $M$. Evidently, $\mathcal C$ and $\mathcal C'$ are isomorphic as posets, and the underlying matroids of $M$ and $M'$ are equal.

%The \word{restriction} of a signed set $X$ to a subset $A \subset E$ is the signed set $X|_A \in \{-,0,+\}^A$ with $(X|_A)_i = X_i$ for all $i \in A$.
Given $F \subset E$ a flat of an oriented matroid $M$ on $E$, the \word{restriction} of $M$ to $F$ and \word{contraction} of $M$ by $F$ are the  oriented matroids $M|_F$ and $M/F$ defined by
\[
  \mathcal C(M|_F) := \{ \pi_F(X) : X \in \mathcal C(M)\} \quad\text{ and }\quad
  \mathcal C(M/F) := \{ X : X \in \mathcal C(M), \; F \subset X^0 \}.
\]
From this description, one sees that the topes of $M/F$ are the covectors $X$ with $X^0 = F$.
The underlying matroids of $M|_F$ and $M/F$ are $\underline M|_F$ and $\underline M/F$, respectively.

\begin{rmk}[Oriented matroids and linear algebra]\label{omgeo}
  The \word{sign map} is $s: \R^E \to \{-,0,+\}^E$ defined by
  \[ s(v)_i = \begin{cases} -, & \pi_i(v)  < 0,\\
                   0, & \pi_i(v) = 0, \\
                   +, & \pi_i(v) > 0.
              \end{cases}
 \]
The sign map explains composition: if $v,w \in \R^E$, then $s(v) \circ s(w) = s(v + \epsilon w)$ for small $\epsilon > 0$.
If $V \subset \R^E$ is a linear subspace, then $\{s(v) : v \in V\}$ is the covectors of an oriented matroid. An oriented matroid $M$ that arises in this way is called \word{realizable}, and $V$ its \word{realization}.
Reorientations of $M$ are obtained by negating some of the coordinate functions on $\R^E$.

By intersecting the coordinate hyperplanes of $\R^E$ with $V$, we obtain a hyperplane arrangement in $V$.
The topes of $M$ correspond to the connected components of the arrangement complement.
More generally, each region of the arrangement is the preimage under $s$ of a covector of $M$.
The poset of the regions' closures, ordered by containment, is isomorphic to the poset of covectors of $M$.
\end{rmk}

\subsection{}
An \word{acyclic flat}\footnote{Acyclic flats may be called ``positive flats'' elsewhere in the literature, e.g. \cite{ARW06, AKW04}.} of an oriented matroid $M$ is a flat $F$ of $M$ such that $0^F+^{E \setminus F}$ is a covector of $M$.
When ordered by containment, the acyclic flats form a lattice $\mathcal L$, called the \word{Las Vergnas face lattice} of $M$.
\begin{prop}\label{acyclic_intervals}
  Let $F$ be a flat of an oriented matroid $M$.
  \begin{enumerate}
  \item $H \supset F$ is an acyclic flat of $M / F$ if and only if $H$ is an acyclic flat of $M$.\label{acyclic_contraction}
  \item If $F$ is an acyclic flat, then $G \subset F$ is an acyclic flat of $M|_F$ if and only if $G$ is an acyclic flat of $M$.\label{acyclic_restriction}
  \end{enumerate}
\end{prop}
\begin{proof}\hfill
  \begin{enumerate}
  \item If $H \supset F$, then $0^H+^{E \setminus H}$ is a covector of $M$ if and only if it is a covector of $M/F$.
  \item Suppose $F$ is an acyclic flat.
    If $G \subset F$ is an acyclic flat of $M$, then $0^G +^{E \setminus G}$ is a covector of $M$, so $0^G +^{F \setminus G}$ is an acyclic flat of $M|_F$.
    
    Conversely, suppose $G$ is an acyclic flat of $M|_F$; in other words, there is a covector $Y$ of $M$ such that $Y^0 \supset G$ and $Y^+ \supset F \setminus G$.
    Since $F$ is an acyclic flat, there is also a covector $X$ of $M$ with $X^0 = F$ and $X^+ = E \setminus F$.
    Their composition satisfies $(X \circ Y)^0 = F \cap G = G$ and 
    \[
       (X\circ Y)^+ = X^+ \cup (Y^+ \setminus X^-) \supset (E \setminus F) \cup ( (F \setminus G) \setminus \emptyset) = E \setminus G,
    \]
    so $G$ is an acyclic flat of $M$.\qedhere
  \end{enumerate}
\end{proof}

When $M$ is realized by $V \subset \R^E$, one can check the acyclicity of a flat $F$ using the equations of $V$.
\begin{prop}\label{poseqn}
  Let $M$ be the oriented matroid of a linear subspace $V \subset \R^E$.
  A flat $F$ is acyclic if and only if there is no linear function $f = \sum_i \alpha_i x_i$ that vanishes on $V$,
  satisfies $\alpha_i \geq 0$ for all $i \in E \setminus F$, and
  has $\alpha_i > 0$ for at least one $i \in E \setminus F$.
\end{prop}
\begin{proof}
  If such an $f$ exists, then $V \cap (0^F \times \R_{>0}^{E \setminus F}) = \emptyset$ because $f$ is strictly positive on $0^F \times \R_{>0}^{E \setminus F}$.
  The converse holds by \cite[Proposition 3.4.8(i) \& (ii)]{bvswz99}, applied to $M/F$.
\end{proof}
\begin{zb}
  \cref{acyclic_intervals}\cref{acyclic_restriction} can fail if $F$ is not an acyclic flat.
  Let $E = \{1,2,3,4\}$, $V \subset \R^E$ be defined by $x_1 - x_2 - x_3 - x_4 = 0$, and $M$ the associated oriented matroid.
  The flats of $M$ are $E$, and all subsets of $E$ of size $\leq 2$.
  In particular, $F = \{1,2\}$ is a flat of $M$, but is not an acyclic flat because the system
  \begin{align*}
  &x_1 = x_2 = 0\\
  &x_1 - x_2 - x_3 - x_4 = 0
  \end{align*}
  has no solutions with $x_3, x_4 >0$.
  For similar reasons, $G = \{1\}$ is a flat, but not an acyclic flat of $M$.

  On the other hand, $G$ \textit{is} an acyclic flat of $M|_F$. This is because the point $(0,2,-1,-1)$, for example,  is in $V$, so $(0,+,-,-)$ is a covector of $M$, so $(0,+)$ is a covector of $M|_F$.
\end{zb}
\begin{rmk}
  If $V \subset \R^n$, then $V \cap \R_{\geq 0}^n$ is a polyhedral cone.
  The face lattice of $V \cap \R_{\geq 0}^n$ is known as the \word{Edmonds-Mandel lattice} of $M$, and the opposite poset is the Las Vergnas face lattice of $M$.
\end{rmk}

A graded poset is \word{thin} if all of its length-two intervals have exactly four elements.
\begin{thm}\cite[Theoerem 4.1.14]{bvswz99}\label{thm:thin}
  The poset of covectors is thin. In particular, the Las Vergnas face lattice is thin.
\end{thm}

\section{Strata of $\mathcal Y_V$}
Let $V \subset \R^E$ be a linear subspace with oriented matroid $M$.
Let $L$ be the set of loops of $M$.
Recall from \cref{sec:mainresult}: the \word{non-negative matroid Schubert variety} $\mathcal Y_V$ is the analytic closure of $V \cap (0^L \times \R^{E \setminus L}_{>0})$ in $(\P_\R^1)^n$.
For each $F \subset G \subset E$, \;$\mathcal Y_{FG}^\circ := \mathcal Y_V \cap (0^F \times \R_{>0}^{G \setminus F} \times \infty^{E \setminus G})$, and $\mathcal Y_{FG} := \clan{\mathcal Y_{FG}^\circ}$.
\footnote{Careful readers will have noticed that in \cref{intro:matroids}, we defined $\mathcal Y_V$ as the closure of $V \cap \R_{\geq 0}^E$. The two definitions agree because $V \cap \R_{\geq 0}^E$ is in the closure of $V \cap (0^L \times \R_{>0}^{E \setminus L})$.}

In this section, we prove  Theorems \ref{schubertthm}\cref{main:strata,main:boundary}, which say that the subsets $\mathcal Y_{FG}^\circ$ indexed by acyclic flats form a stratification of $\mathcal Y_V$, and that the closure poset is isomorphic to the interval poset of the Las Vergnas face lattice of the oriented matroid of $V$.
\begin{lem}\label{Vbarpos}
  Let $V \subset \R^E$ be a linear subspace defining an oriented matroid $M$.
  If $F \subset G \subset E$ are flats of $M$, then
  \[ Y_V \cap (0^F \times \R_{>0}^{G\setminus F} \times \infty^{E \setminus G}) =  \big(\pi_G(V) \cap (0^F \times \R_{>0}^{G\setminus F})\big) \times \infty^{E \setminus G}. \]
  In particular, $Y_V \cap (0^F \times \R_{>0}^{G \setminus F} \times \infty^{E \setminus G})$ is nonempty if and only if $F$ is an acyclic flat of $M|_G$.
\end{lem}
\begin{proof}
  The equality follows from \cref{schubertfacts}\cref{ziso}.
  For the statement on non-emptiness, recall (from \cref{bg:oriented}) that the oriented matroid of $\pi_G(V)$ is $M|_G$.
  Non-emptiness of $\pi_G(V) \cap (0^F \times \R_{>0}^{G \setminus F})$ is equivalent to $0^F \times +^{G \setminus F}$ being a covector of $M|_G$, in turn equivalent to acyclicity of $F$ in $M|_G$.
\end{proof}
\begin{lem}\label{infinity_avoidance}
  Let $V \subset \R^E$ be a linear subspace defining an oriented matroid $M$. Let $F \subset E$.
  If $F$ is not an acyclic flat, then $\mathcal Y_V \cap (\R_{\geq 0}^F \times \infty^{E \setminus F}) = \emptyset$.
  \end{lem}
\begin{proof}
  If $F$ is not a flat, then $\mathcal Y_V \cap (\R_{\geq 0}^F \times \infty^{E \setminus F}) = \emptyset$ by \cref{schubertfacts}\cref{zstrata}. 
    Otherwise, suppose $F$ is a flat, but not an acyclic flat and let $w \in \R_{\geq 0}^F \times \infty^{E \setminus F}$.
    By \cref{poseqn}, there is a linear functional $f = \sum_i \alpha_i x_i$ that vanishes on $V$ and satisfies $\alpha_i \geq 0$ for all $i \not\in F$, with at least one such $\alpha_i$ nonzero.
        When  $N \gg 0$, $f$ does not vanish at any point of $\prod_{i \in F} (w_i - \frac1N, w_i + \frac1N) \times \prod_{j \in E \setminus F} (N,\infty)$. Hence, the neighborhood $\prod_{i \in F} (w_i - \frac1N,w_i+\frac1N) \times \prod_{j \in E \setminus F} (N,\infty]$ of $w$ does not intersect $V \cap \R^E_{>0}$, meaning that $w \not \in \mathcal Y_V$.
  \end{proof}
\begin{lem}\label{geq0iso}
  Let $V \subset \R^E$ be a linear space defining an oriented matroid $M$.
  If $F  \subset E$ is an acyclic flat, then
  \[
    Y_V \cap (\R_{\geq 0}^F \times \infty^{E \setminus F})
     =\mathcal Y_V \cap (\R_{\geq 0}^F \times \infty^{E \setminus F}).
   \]
   % and $\pi_F$ restricts to an isomorphism
   % \[
   %   \mathcal Y_V \cap (\R_{\geq 0}^F \times \infty^{E \setminus F}) \to \pi_F(V) \cap \R_{\geq 0}^F.
   %   \]
\end{lem}
\begin{proof}
    Let $w \in Y_V \cap (\R_{\geq 0}^F \times \infty^{E \setminus F})$.
    By \cref{schubertfacts}\cref{zstrata}, there is $w' \in V$ such that $\pi_F(w') \times \infty^{E \setminus F} = w$.
    Since $F$ is acyclic, there is also $u \in V \cap (0^F \times \R_{>0}^{E \setminus F})$.
    For large $t > 0$, $w' + tu \in V_{\geq 0}$, and $\lim_{t \to \infty} w' + tu = w$,
    so $w \in \clan{V \cap \R_{\geq 0}^E} = \clan{V \cap \R_{>0}^E} = \mathcal Y_V$.
    This shows
    \[ Y_V \cap (\R_{\geq 0}^F \times \infty^{E \setminus F}) \subset \mathcal Y_V \cap (\R_{\geq 0}^F \times \infty^{E \setminus F}),\]
    and the reverse inclusion is obvious, so the two sets are equal.
    % For the isomorphism: by \cref{schubertfacts}(\ref{ziso}), $\pi_F$ restricts to an isomorphism $Y_V \cap (\R^F \times \infty^{E \setminus F}) \to \pi_F(V)$. Further restriction to $Y_V \cap (\R^F_{\geq 0} \times \infty^{E \setminus F})$ yields an isomorphism
    % \[ \mathcal Y_V \cap (\R_{\geq 0}^F  \times \infty^{E \setminus F}) =  Y_V \cap (\R_{\geq 0}^F \times \infty^{E \setminus F})  \to \pi_F(V) \cap \R_{\geq 0}^F. \]
  \end{proof}
  We are now ready to prove the first part of the main result.
\begin{proof}[Proof of \cref{schubertthm}\cref{main:strata}]
  If $G$ is not an acyclic flat, then $\mathcal Y_{FG}^\circ$ is empty by \cref{infinity_avoidance}.
  If $G$ \textit{is} an acyclic flat, but $F$ is not, then $F$ is not acyclic in $M|_G$ by \cref{acyclic_restriction}, so $\mathcal Y_{FG}^\circ$ is empty by \cref{Vbarpos}.

  Conversely, if both $F$ and $G$ are acyclic flats, then
  \[
    \mathcal Y_{FG}^\circ
    = Y_V \cap (0^F \times \R_{>0}^{G \setminus F} \times \infty^{E \setminus F})
    = \big(\pi_G(V) \cap (0^F \times \R_{>0}^{G \setminus F})\big) \times \infty^{E \setminus G}
  \]
  by \cref{Vbarpos} and \cref{geq0iso}.
  Consequently, $\mathcal Y_{FG}^\circ$ is the interior of a polyhedral cone of dimension $\rk(G) - \rk(F)$.
  Via the equalities
  \[
    Y_{FG}^\circ
    = Y_V \cap (0^F \times \R_{\neq 0}^{G \setminus F} \times \infty^{E \setminus G})
    = \big(\pi_G(V) \cap (0^F \times \R_{\neq 0}^{G \setminus F})\big) \times \infty^{E \setminus G},
  \]
  we see $Y_{FG}^\circ$ is the complement of a real hyperplane arrangement in $V \cap \{x_i = 0 : i \in F\}$.
  Since $F$ and $G$ are acyclic, $0^F +^{G \setminus F}$ is a tope of the oriented matroid associated to this arrangement; the corresponding connected component of the arrangement complement is $\mathcal Y_{FG}^\circ$.
\end{proof}

The following two corollaries of \cref{schubertthm}\cref{main:strata}  provide geometric interpretations for restriction and contraction at the level of totally non-negative matroid Schubert varieties. They closely resemble \cref{schubertfacts}\cref{zstrata} and \cref{zcontract}.
\begin{cor}\label{aniso}
  Let $V \subset \R^E$ be a linear subspace defining an oriented matroid $M$.
  If $G \subset E$ is an acyclic flat of $M$,
  then
  \[
    \mathcal Y_V \cap ((\P^1)^G \times \infty^{E \setminus G}) = 
    \mathcal Y_{\pi_G(V)} \times \infty^{E \setminus G}.
  \]
\end{cor}
\begin{proof}
  By \cref{geq0iso} and \cref{Vbarpos}, $\mathcal Y_V \cap ((\P^1)^G \times \infty^{E \setminus G})$ contains
  $(\pi_G(V) \cap \R_{>0}^G) \times \infty^{E \setminus G}$, the closure of which is $\mathcal Y_{\pi_G(V)} \times \infty^{E \setminus G}$. This proves the ``$\supset$'' containment.
  For the reverse: by \cref{schubertthm}\cref{main:strata} the nonempty strata of $\mathcal Y_V$ are of the form $\mathcal Y_V \cap (0^F \times \R_{>0}^{G \setminus F} \times \infty^{E \setminus G})$, with $F \subset G$ acyclic flats of $M$.
  By \cref{acyclic_intervals}\cref{acyclic_restriction}, $F$ and $G$ are also acyclic flats of $M|_G$, the oriented matroid represented by $\pi_G(V)$.
  Hence, 
  \[
    \mathcal Y_V \cap (0^F \times \R_{>0}^{G \setminus F} \times \infty^{E \setminus G})
    = \big(\pi_G(V) \cap (0^F \times \R_{>0}^{G \setminus F})\big) \times \infty^{E \setminus G}
    = \big(\mathcal Y_{\pi_G(V)} \cap (0^F \times \R_{>0}^{G \setminus F})\big) \times \infty^{E \setminus G}
  \]
  by \cref{geq0iso} and \cref{Vbarpos}, which completes the proof.
\end{proof}
A proof along the same lines shows:
\begin{cor}\label{lem:contraction}
  Let $V \subset \R^E$ be a linear subspace defining an oriented matroid $M$. If $F$ is an acyclic flat, then
  $\mathcal Y_V \cap (0^F \times (\P^1)^{E \setminus F}) = \mathcal Y_{V \cap \ker(\pi_F)}$.
\end{cor}
Together, these corollaries yield a short proof of the main result's second part.
\begin{proof}[Proof of \cref{schubertthm}\cref{main:boundary}]
  By \cref{aniso} and \cref{lem:contraction}, $\mathcal Y_{FG}$ is equal to $0^F \times \mathcal Y_{\pi_G(V \cap \ker(\pi_F))} \times \infty^{E \setminus G}$,
  in turn the closure of $0^F \times \big(\pi_G(V \cap \ker(\pi_F)) \cap (0^F \times \R_{>0}^{G \setminus F})\big) \times \infty^{E \setminus G}$.
  The latter set is equal to $\mathcal Y_{FG}^\circ$.
  Strata of $\mathcal Y_{\pi_G(V \cap \ker(\pi_F))}$ correspond to pairs $F' \subset G'$ of acyclic flats of $(M/F)|_G$.
  By \cref{acyclic_intervals}, such $F' \subset G'$ are precisely the acyclic flats of $M$ such that $[F', G'] \subset [F,G]$, as desired.
\end{proof}

\section{Topology of $\mathcal Y_V$}\label{sec:topology}
In this section, we prove \cref{schubertthm}\cref{main:ball}, which says that $\mathcal Y_V$ is a regular CW complex homeomorphic to a closed Euclidean ball.
For basics on CW complexes, we refer the reader to \cite{LW69}.

\subsection{Shellings and topology}
A CW complex is \word{regular} if the closure of any of its cells is homeomorphic to a closed Euclidean ball.
A \word{$d$-complex} is a finite regular CW complex with all maximal cells of dimension $d$.
Maximal closed cells of a $d$-complex $\Delta$ are \word{facets}.
Following \cite{B84} or \cite[Appendix 4.7]{bvswz99}, a \word{shelling} of $\Delta$ is an ordering of its facets $(F_1, \ldots, F_m)$ such that the boundary complex of $F_1$ has a shelling, $F_j \cap (\cup_{i=1}^{j-1} F_i)$ is $(d-1)$-complex for $1 < j \leq m$, and the boundary of $F_j$ has a shelling in which the facets of $F_j \cap (\cup_{i=1}^{j-1} F_i)$ come first for $1 < j \leq m$.

\begin{zbthm}\cite[Theorem 4.5]{BW83}\label{polytopes}
  The boundary complex of any convex polytope is shellable.
\end{zbthm}

A shellable complex satisfies the so-called ``Property S'' of \cite{BW83}. It is equivalent to shellability for simplicial complexes.
\begin{prop}\label{propertys}
  If $(F_1, \ldots, F_m)$ is a shelling of a $d$-complex $\Delta$, then 
  for all $i > j$ there exists $k < i$ such that $F_i \cap F_j \subset F_k$ and $F_k \cap F_j$ has dimension $d-1$.
\end{prop}
\begin{proof}
  If $i > j$ then each cell of $F_i \cap F_j$ is contained in a cell $G$ of $\Delta$, maximal among those contained in $C_i := F_i \cap (F_1 \cup \cdots \cup F_{i-1})$.
  Since $G$ cannot be written as a union of its proper faces, it must be contained in some $F_k$ with $k < i$. The dimension of $G$ is $d-1$ because $C_i$ is pure of dimension $d-1$.
\end{proof}

The following result is our main topological tool.
\begin{prop}\cite[Proposition 4.3]{B84}\label{shellabletop}
  A shellable $d$-complex is homeomorphic to a closed Euclidean ball if each of its $(d-1)$-cells is a face of at most two $d$-cells, and some $(d-1)$-cell is contained in only one $d$-cell.
\end{prop}

\subsection{Proof of \cref{schubertthm}\cref{main:ball}}
Let $\mathcal Y_0$ be the set of all points in $\mathcal Y_V$ with at least one coordinate zero, and fix $w \in \mathcal Y_{\emptyset E}^\circ$.
Define
\[ \mu: (\R_{\geq 0} \cup \infty)^n \to [0,1], \quad (y_1, \ldots, y_n) \mapsto 1 - \exp(-\min_i \{ y_i / w_i\}) \]
The value $\mu(y)$ is the largest value of $t \in [0,1]$ such that $y - \ln(1-t) w$ has non-negative coordinates.
The map
\[ \psi: \mathcal Y_V \to \mathcal Y_0  \times [0,1] / (\mathcal Y_0 \times \{1\}), \quad y \mapsto (y + \ln(1 - \mu(y))w, \mu(y)) \ \]
is a (non-cellular) homeomorphism, with inverse $(x,t) \mapsto x - \ln(1-t)w$.
To show $\mathcal Y$ is a ball, it now suffices to show that $\mathcal Y_0$ is a ball.
This follows from \cref{shellabletop} and \cref{shelling}, whose proof is below.
By induction, it is now proved that all closed strata of $\mathcal Y_V$ are homeomorphic to closed balls, so $\mathcal Y_V$ is a regular CW complex.

\begin{lem}\label{shelling}
  $\partial \mathcal Y_V := \mathcal Y_V \setminus \mathcal Y_{\emptyset,E}$ is a regular shellable CW complex in which all $(d-1)$-cells are contained in exactly two $d$-cells.
   The shelling can be chosen such that all the cells of $\mathcal Y_0$ come first, and there are $(d-1)$-cells of $\mathcal Y_0$ that are contained in only one $d$-cell of $\mathcal Y_0$.
\end{lem}
\begin{proof}
  Let $M$ be the oriented matroid of $V$, $\dim V = d+1$, and $\mathcal L$ its Las Vergnas face lattice.
  We first check the containment assertions.
  The statement that every $(d-1)$-cell is contained in exactly two $d$-cells follows easily from the description of cells given by \cref{schubertthm}, combined with \cref{thm:thin}, which says that $\L$ is thin.

  The $d$-cells of $\mathcal Y_0$ are $\mathcal Y_{F,E}$ with $F$ of rank 1.
  If $G$ is a corank 1 acyclic flat containing $F$, then $\mathcal Y_{F,E} \supset \mathcal Y_{F,G}$, but no other $d$-cell of $\mathcal Y_0$ contains $\mathcal Y_{F,G}$.

  \smallskip
  Next, we check the shellability statements.
  The Las Vergnas face lattice of $M$ is dual to the face poset of the polyhedral cone $V_{\geq 0}$; therefore, $\mathcal L^{op}$ and $\mathcal L$ are the face posets of convex polytopes $P^{op}_M$ and $P_M$ with facets in bijection with the rank 1 and corank 1 acyclic flats of $M$, respectively.
  By \cref{polytopes}, let $(F_1, \ldots, F_s)$ and $(G_1, \ldots, G_t)$ be shellings of $P^{op}_M$ and $P_M$, respectively.
  We will show by induction on $d$ that $([F_1,E], \ldots, [F_s,E], [\emptyset,G_1], \ldots, [\emptyset,G_t])$ indexes a shelling of $\partial \mathcal Y_V$.

  The statement holds when $d=1$; suppose $d > 1$.
  The boundary of $\mathcal Y_{F_1, E} \cong \mathcal Y_{V \cap \{ x_i = 0 : i \in F_1\}}$ is shellable by the induction hypothesis.
  For later cells, we break into two cases.
  First consider
  \[
    C_j := \mathcal Y_{F_j, E} \cap ( \cup_{i < j} \mathcal Y_{F_i,E}) = \cup_{i < j} \mathcal Y_{F_i \vee F_j, E}.
  \]
  Since $(F_1, \ldots, F_s)$ is a shelling of $P^{op}_M$, for each $i < j$, there is $k$ such that $\mathcal Y_{F_i \vee F_k, E} \supset \mathcal Y_{F_i \vee F_k, E}$ and $F_i \vee F_k$ has rank 2 by \cref{propertys}.
  This shows $C_j$ is a $(d-1)$-complex.

  Let $P^{op}_M(F)$ be the face of $P_M^{op}$ corresponding to an acyclic flat $F$ in $\mathcal L^{op}$.
  By hypothesis, $P_M^{op}(F_j)$ has a shelling in which the facets $P_M^{op}(F_j \vee F_i)$, $i < j$ and $\rk(F_j \vee F_i) = 2$ come first.
  The face poset of $P_M^{op}(F_j)$ is the same as that of $P_{M/F_j}^{op}$, the polytope associated to the oriented matroid of $V \cap \{x_j = 0\}$. Hence, by induction $\partial \mathcal Y_{V \cap \{x_j = 0\}} \cong \partial \mathcal Y_{F_j,E}$ has a shelling in which the $(d-1)$-cells of $C_j$ come first.

  We now consider 
  \[
    D_j := \mathcal Y_{\emptyset, G_j} \cap \big( \mathcal Y_0 \cup (\cup_{i < j} \mathcal Y_{\emptyset F_i})\big) = (\cup_{F_k \subset G_j} \mathcal Y_{F_k, G_j}) \cup (\cup_{i < j} \mathcal Y_{\emptyset, G_i \cap G_j}).
  \]
  All cells of the form $\mathcal Y_{F_k,G_j}$ are dimension $d-1$, and $\cup_{i<j} \mathcal Y_{\emptyset, G_i \cap G_j}$ is a $(d-1)$-complex by \cref{propertys}, as above, so $D_j$ is a $(d-1)$-complex.
  Observing that $\mathcal Y_{\emptyset,G_j} \cong \mathcal Y_{\pi_{G_j}(V)}$, shellability follows as above.
\end{proof}
\begin{rmk}
  Our proof relies on the fact that both $\mathcal L$ and $\mathcal L^{op}$ are face lattices of polytopes, hence CL-shellable (see \cite{BW83}).
  It is known that $\mathcal L$ is CL-shellable even when $M$ is not realizable \cite[Theorem 4.3.5]{bvswz99}, but remains open whether $\mathcal L^{op}$ is.
\end{rmk}
\begin{rmk}
  A slightly different route to \cref{schubertthm}\cref{main:ball}: the \word{order complex} of a poset is the simplicial complex whose faces are chains in the poset, and a poset is \word{shellable} if its order complex is.
  By \cite[Theorem 4.3.5]{bvswz99}, $\mathcal L^{op}$ is shellable, so $\mathcal L$ is also shellable, so the interval poset of $\mathcal L$ is shellable by \cite[Theorem 8.5]{BW83}.
  Every length 2 interval in the  interval poset of $\mathcal L$ has cardinality 4, so $\partial \mathcal Y_V$ is homeomorphic to a sphere by \cite[Propositions 1.1, 1.2]{DK73} and \cite[Theorem III.1.7]{LW69}.
  In fact, by \cite[Proposition 4.7.26]{bvswz99}, $\partial \mathcal Y_V$ is a PL sphere.
  The link of a vertex of a PL sphere is also a PL sphere; in particular, the equator $\mathcal Y_{\ominus} := \mathcal Y_0 \setminus V_{\geq 0}$ is a PL sphere because it is the link of $\mathcal Y_{E,E}$.
  The star of a point is a cone over its link, so $\mathcal Y_0$ is a cone over $\mathcal Y_\ominus$, so $\mathcal Y_0$ is a closed ball. The proof may now be completed as above.
\end{rmk}
\section{Topology of $Y_V$}
Let $V \subset \R^E$.
In this section, we will show that $Y_V$ admits a regular cell decomposition.

\smallskip
We first record a consequence of \cref{schubertthm}.
Let $M$ be the oriented matroid of $V$, and $s: \R^E \to \{-,0,+\}^E$ the sign map (see \cref{omgeo}).
Fix a tope $T$ of $M$.
A flat is \word{relatively acyclic} in $T$ if it is the zero set of a covector contained in $T$.
Define ${}_T\mathcal Y_V := \clan{s^{-1}(T)}$, the analytic closure of $s^{-1}(T)$ in $(\P_\R^1)^E$.
For each pair of sets $F \subset G \subset E$, let
\[ {}_T \mathcal Y_{FG}^\circ := {}_T \mathcal Y_V \cap (0^F \times \R_{>0}^{(G \setminus F) \cap T^+} \times \R_{<0}^{(G \setminus F) \cap T^-} \times \infty^{E \setminus G}). \]
Finally, set ${}_T \mathcal Y_V := \clan{{}_T\mathcal Y_V^\circ}$.

\begin{cor}\label{topecor}
  Fix a tope $T$ of the oriented matroid of $V \subset \R^E$.
  Then
  \begin{enumerate}
  \item ${}_T\mathcal Y_{FG}^\circ$ is nonempty if and only if $F \subset G$ are acyclic flats in $T$.
  In this case, ${}_T\mathcal Y_{FG}^\circ$ is a single connected component of $Y_{FG}^\circ$, and is a semi-algebraic cell isomorphic to $(\R_{>0})^{\rk(G) - \rk(F)}$. \label{topecor:strata}
\item The closure ${}_T\mathcal Y_{FG}$ of a nonempty cell ${}_T\mathcal Y_{FG}^\circ$ decomposes as the disjoint union of cells ${}_T\mathcal Y_{F',G'}^\circ$ with $F \subset F' \subset G' \subset G$. \label{topecor:boundary}
\item This decomposition makes ${}_T\mathcal Y_V$ a shellable regular CW complex homeomorphic to a closed ball.\label{topecor:ball}
  \end{enumerate}
\end{cor}
\begin{proof}
  For $A \subset E$, let $-_A: (\P_\R^1)^E \to (\P_\R^1)^E$ be the map that negates the coordinates indexed by $E$.
  The result follows from \cref{schubertthm} because $-_{T^-}(s^{-1}(T)) = -_{T^-}(V) \cap \R_{\geq 0}^E$.
\end{proof}
\begin{rmk}
  A tope in the matroid-theoretic setting is akin to a \word{pinning} in the Lie-theoretic setting, as defined in \cite{Lus-1}.
  Indeed, $\SL(2)$ has just one negative simple root.
  Choosing an isomorphism $y:\R\to U_{-\alpha}$ up to positive scalars in each factor of $\SL(2)^n$ is the same as choosing which side of $\R \subset \P^1_\R$ will be regarded as positive, hence is the same as choosing a positive side for each hyperplane in $V \subset \R^n$ obtained by intersecting $V$ with a coordinate hyperplane of $\R^n$.
\end{rmk}

The various subsets ${}_T \mathcal Y_V$ are not disjoint. The following statement characterizes their intersections.
\begin{lem}\label{lem:cellequiv}
  Let $M$ be the oriented matroid of $V \subset \R^E$.
  Define an equivalence relation on the set of all triples $(F,G,T)$, with $F \subset G$ flats relatively acyclic in the tope $T$ of $M$, by $(F,G,T) \sim (F',G',T')$ if and only if $(F,G) = (F',G')$ and $\pi_{G \setminus F}(T) = \pi_{G \setminus F}(T')$.
  The intersection of ${}_T \mathcal Y_{FG}^\circ$ and ${}_{T'} \mathcal Y_{F'G'}^\circ$ is empty unless $(F,G,T) \sim (F',G',T')$, in which case ${}_T \mathcal Y_{FG}^\circ = {}_{T'} \mathcal Y_{F'G'}^\circ$.
\end{lem}
\begin{proof}
    Reorienting, applying \cref{Vbarpos} and \cref{geq0iso}, then reverting to the original orientation, we see 
  \begin{align}
    {}_T\mathcal Y^\circ_{FG} &= (\pi_G(V) \cap (0^F \times \R_{>0}^{T^+ \cap (G \setminus F)} \times \R_{<0}^{T^- \cap (G \setminus F)})) \times \infty^{E \setminus G} \quad\text{and}\label{expansion}\tag{$*$}\\
        {}_{T'}\mathcal Y^\circ_{F'G'} &= (\pi_{G'}(V) \cap (0^{F'} \times \R_{>0}^{{T'}^+ \cap (G' \setminus F')} \times \R_{<0}^{{T'}^- \cap (G' \setminus F')})) \times \infty^{E \setminus G'}.\nonumber
  \end{align}
  The result follows.
\end{proof}

For each pair of flats $F \subset G \subset E$ and tope $T$ of $(M/F)|_G$, let $Y_{FGT}^\circ$ be the connected component of $Y_{FG}^\circ$ corresponding to the tope $T$.
Explicitly,  $Y_{FGT}^\circ := Y_V \cap (0^F \times \R_{>0}^{T^+} \times \R_{<0}^{T^-} \times \infty^{E \setminus G})$.
As usual, let $Y_{FGT} := \clan{Y_{FGT}^\circ}$.
\begin{lem}\label{lem:othercellequiv}
  The equivalence classes of $\sim$ (defined as in \cref{lem:cellequiv}) are in bijection with cells $Y_{FGT}^\circ$.
  If $[S,I,J]$ is the equivalence classes corresponding to $Y_{FGT}^\circ$, then $Y_{FGT}^\circ = {}_S \mathcal Y_{IJ}^\circ$.
   Explicitly, $Y_{FGT}^\circ = {}_S \mathcal Y_{IJ}^\circ$ if and only if $(F,G) = (I,J)$ and $\pi_{G \setminus F}(S) = \pi_{G \setminus F}(T)$.
\end{lem}
\begin{proof}
  Given ${}_S \mathcal Y_{IJ}^\circ$, take $F := I$, $G := J$, and $T := \pi_G(X)$, where $X$ is the unique covector contained in $S$ with $X^0 = F$.
  Evidently, $Y_{FGT}$ is independent of the representative of $[S,I,J]$.

  For the inverse: let $F \subset G$ be flats of $M$ and let $T$ be a tope of $(M/F)|_G$.
  There are covectors $\tilde F, \tilde G$ of $M$ satisfying $\tilde F^0 = F$, $\pi_G(\tilde F) = T$, and $\tilde G^0 = G$.
  The composition $X := \tilde G \circ \tilde F$ then satisfies $\tilde G \leq X$, $X^0 = F$, and $\pi_G(X) = T$.
  Both $F$ and $G$ are relatively acyclic with respect to any tope $S \geq X$ of $M$, so there is an equivalence class $[S,F,G]$. This class is independent of $S$ because $\pi_{G \setminus F}(S) = \pi_{G \setminus F}(T)$.

  Under the bijection described above, we see that $Y_{FGT}$ corresponds to $[S,I,J]$ if and only if $(F,G) = (I,J)$ and $\pi_{G \setminus F}(S) = \pi_{G \setminus F}(T)$.
  In this case, ${}_S \mathcal Y_{IJ}^\circ = Y_{FGT}$ by \cref{expansion}.
\end{proof}
\begin{cor}\label{YV}
  Let $M$ be the oriented matroid of $V \subset \R^E$.
  \begin{enumerate}
  \item $Y_{FGT}$ contains $Y_{F'G'T'}^\circ$ if and only if $F \subset F' \subset G' \subset G$ and there is a tope $S$ of $M$ satisfying: $F,G,F',G'$ are all relatively acyclic in $S$, $\pi_{G\setminus F}(S) = \pi_{G \setminus F}(T)$ and $\pi_{G'\setminus F'}(S) = \pi_{G'\setminus F'}(T')$.\label{YVcontainment}
  \item The cells $Y_{F,G,T}^\circ$,  where $F \subset G$ run over flats of $M$ and $T$ runs over topes of $(M/F)|_G$, form a regular CW decomposition of $Y_V$.\label{YVregular}
  \end{enumerate}
\end{cor}
\begin{proof}
  We prove the second statement first.
  The set $\cup_T\, {}_T\mathcal Y_V$ is closed in $(\P_\R^1)^E$ and contains $V$; therefore, it is equal to $Y_V$.
 Together with \cref{lem:cellequiv} and \cref{topecor}, this implies the collection $\{{}_S \mathcal Y_{IJ}^\circ\}_{I,J,S}$ is a regular cell decomposition of $Y_V$.
 The cells in this decomposition are in fact the sets $Y_{FGT}^\circ$ by \cref{lem:othercellequiv}, completing the proof of \cref{YV}\cref{YVregular}.

 We now prove the first statement.
The closure of any cell is contained in some set ${}_S \mathcal Y_V$.
Hence, the closure of $Y_{FGT}^\circ$ contains $Y_{F'G'T'}^\circ$ if and only if there is a tope $S$ of $M$ such that  $Y_{FGT}^\circ = {}_S \mathcal Y^\circ_{FG}$, $Y_{F'G'T'}^\circ = {}_S \mathcal Y^\circ_{F'G'}$, and ${}_S \mathcal Y_{FG} \supset {}_S \mathcal Y^\circ_{F'G'}$.
By \cref{topecor} and \cref{lem:othercellequiv}, this is equivalent to the conditions specified by \cref{YV}\cref{YVcontainment}.
\end{proof}

\begin{rmk}
  If $V \subset \R^3$ is defined by $x_1 + x_2 - x_3 = 0$, then $Y_V$ has nontrivial first homology.
  This means $Y_V$ is not a shellable cell complex, since a shellable $d$-complex always has the homotopy type of a wedge of $d$-spheres \cite[Proposition 4.3]{B84}.
\end{rmk}
\printbibliography
\end{document}